\documentclass[11pt]{amsart}

\usepackage[a4paper,margin=1in]{geometry}
\usepackage{amsmath,amssymb,amsthm,mathtools}
\usepackage{enumitem}
\usepackage{hyperref}
\usepackage{mathrsfs}
\usepackage[T1]{fontenc}
\usepackage[utf8]{inputenc}
\usepackage{tikz}
\usetikzlibrary{arrows.meta,positioning,calc}

\newtheorem{theorem}{Theorem}[section]
\newtheorem{lemma}[theorem]{Lemma}
\newtheorem{corollary}[theorem]{Corollary}

\newtheorem{conjecture}[theorem]{Conjecture}
\theoremstyle{definition}
\newtheorem{definition}[theorem]{Definition}

\newtheorem{remark}[theorem]{Remark}

\newcommand{\rhoK}{\rho}
\newcommand{\seed}{s}

\title{A Degree-Preserving Builder--Chooser Game}
\author{András London}

\address{Institute of Informatics, University of Szeged, Szeged, Hungary}
\email{london@inf.u-szeged.hu \\ ORCID: 0000-0003-1957-5368}
\date{}

\begin{document}
\maketitle

\begin{abstract}
We propose a degree-preserving variant of the Builder--Chooser clique game of Pettie, Tardos, and Walczak \cite{PTW}.
In each round, Builder chooses a matching, performs a degree-preserving growth (DPG) step by replacing the chosen edges with edges incident to a new vertex. Then partitions the entire edge set into two parts, and Chooser keeps one part.
We begin the study of this game with the first nontrivial target, namely forcing a triangle.
For triangle-free initial graphs we prove an exact one-round criterion, derive an exact one-round threshold on paths and exact forcing times on cycles, and identify the $5$-cycle as the first genuine two-round example.
We then formulate a one-round criterion for larger cliques, prove a sharp exact result for forcing $K_4$ from triangle-free seeds. We establish general lower bounds on clique-forcing times from clique-free seeds, and isolate a conjectural template-amplifier lemma which, if proved, would imply that every clique is forceable from some triangle-free seed.
\end{abstract}

\section{Introduction}

The clique game of Pettie, Tardos, and Walczak is motivated by the Erd\H{o}s--Hajnal conjecture on high-chromatic high-girth subgraphs and, more broadly, by the principle that Builder strategies in suitable graph games can be converted into structural constructions \cite{EH,PTW}.
In the present note we propose a different motivation for a related game. The motivation for our game is primarily mathematical.
The Builder--Chooser clique game of Pettie, Tardos, and Walczak asks how clique structure can
be forced when Builder grows the graph and Chooser repeatedly keeps only one side of a binary
partition of the edge set. We study a more rigid variant in which Builder no longer has free control over the neighborhood of the new vertex.
Instead, growth must proceed through a degree-preserving matching step.

This viewpoint is suggested by the degree-preserving growth (DPG) process. In that framework, vertex degree is treated as an intrinsic property of a vertex, and the growth step respects existing degree constraints by replacing a matching with incidence to a new vertex. Such a rule is natural in settings where old vertices have limited connection capacity or where adding new incident edges carries a local cost. The DPG literature motivates this perspective both from network modeling and from the
mathematical theory of degree-preserving graph dynamics \cite{KharelMezeiToroczkai,ErdosKharelMezeiToroczkai}.

Our game formalizes exactly this situation.
Builder performs a degree-preserving local rewrite by choosing a matching, deleting it, adding a new vertex, and joining the new vertex to the endpoints of the matching. He then partitions the entire edge set into two parts, and Chooser keeps one part. The resulting forcing time of a target graph $H$ measures the robustness of the emergence of $H$ under constrained growth and worst-case binary pruning.

This restriction leads to a qualitatively different problem than the Builder--Chooser clique game. In the original game, the new vertex may be attached very freely and connected to all existing vertices, so clique growth is driven by global branching.
In our variant, clique growth is mediated by a local matching-based rewrite, which creates strong structural bottlenecks.
The resulting question is therefore whether clique forcing is still possible under such a rigid growth mechanism, and if so, at what cost in forcing time ($\tau$) and in the complexity of the initial seed (a graph $G$).

To state this program quantitatively, for $k\ge 3$ we define
\[
\rhoK(k):=\min\{\tau_{K_k}(G): G \text{ is triangle-free and } \tau_{K_k}(G)<\infty\},
\]
with $\rhoK(k)=\infty$ if no triangle-free graph can force $K_k$, and
\[
\seed(k):=\min\{|V(G)|: G \text{ is triangle-free and } \tau_{K_k}(G)<\infty\},
\]
with $\seed(k)=\infty$ if no such graph exists.
Thus $\rhoK(k)$ measures the minimum forcing time and $\seed(k)$ is the minimum size of a
triangle-free seed from which Builder can force a copy of $K_k$.

The first nontrivial target is the triangle $K_3$, which can be interpreted as the smallest feedback motif, closure pattern, or locally redundant structure.
Thus, unlike the original clique game, whose emphasis is on forcing large complete subgraphs under a very flexible attachment rule, our game asks how clique structure from triangles
to larger complete graphs can be forced in a sparse evolving graph when growth is constrained by a local degree-preserving operation.

This game belongs to the broader landscape of positional and graph games. Classical references are Beck's monograph on combinatorial games \cite{Beck}, and the foundational work on biased positional by Chv\'atal and Erd\H{o}s \cite{ChvatalErdos}. The book of
Hefetz, Krivelevich, Stojakovi\'c, and Szab\'o  place clique-building, Maker--Breaker,
Avoider--Enforcer, and related graph-forcing problems into a common framework \cite{HKSS, HefetzKrivelevichSzabo2007}.
More specific clique- and graph-oriented directions include, for instance, Gebauer's clique game \cite{GebauerClique},
and recent multistage positional games \cite{BarkeyClemensHamannMikalackiSgueglia}.
Within that landscape, the game of Pettie, Tardos, and Walczak is the closest direct ancestor of our model, since it also combines Builder growth with an adversarial binary choice on the edge set.
At the same time, the present degree-preserving variant is structurally quite different.
In this sense, our game may be viewed as a constrained positional graph game in which clique forcing is studied under a rigid evolving-host rule.

The manuscript has three parts. First, we develop an exact theory for triangle forcing on triangle-free graphs, including an exact one-round threshold on paths and exact forcing times on cycles.
Second, we turn to larger cliques: we obtain a one-round criterion for forcing $K_{r+1}$,
prove the exact value $\rhoK(4)=3$ for forcing $K_4$ from triangle-free seeds, and derive general lower bounds for clique forcing from $K_r$-free graphs.
Third, we formulate a recursive universality program for forcing arbitrary cliques, identify a precise missing ``template amplifier'' lemma, and show that such an amplifier
would imply finiteness of $\rhoK(k)$ for all $k$ together with explicit upper bounds.

\section{The game}

\begin{definition}[Degree-preserving Builder--Chooser game]
Let $G_0$ be a finite simple graph.
At round $t$, the current graph is $G_t=(V_t,E_t)$.
The round consists of the following steps.
\begin{enumerate}[label=(\roman*)]
    \item \textbf{Builder's growth step.}
    Builder chooses a matching
    \[
        M_t=\{x_1y_1,\dots,x_my_m\}\subseteq E_t.
    \]
    Let $V(M_t)$ denote the set of endpoints of the edges of $M_t$.
    Builder forms an intermediate graph $G_t^+$ by deleting the edges of $M_t$,
    adding a new vertex $w_t$, and joining $w_t$ to every vertex of $V(M_t)$:
    \[
        G_t^+ := G_t - M_t + \{w_tz:\ z\in V(M_t)\}.
    \]

    \item \textbf{Builder's partition step.}
    Builder chooses a partition
    \[
        E(G_t^+) = E_t^{(0)} \sqcup E_t^{(1)}.
    \]

    \item \textbf{Chooser's move.}
    Chooser selects $i_t\in\{0,1\}$, and the next graph is
    \[
        G_{t+1}:=(V(G_t^+),E_t^{(i_t)}).
    \]
\end{enumerate}
Vertices are never deleted, only edges.
\end{definition}

\begin{definition}[Forcing time]
For a fixed graph $H$, we say that $G_0$ is \emph{$H$-forceable} if Builder has a strategy
ensuring that $H\subseteq G_t$ for some $t\ge 0$.
We write
\[
\tau_H(G_0):=\min\{t\ge 0:\text{Builder can force }H\subseteq G_t \text{ from host graph } G_0\},
\]
with $\tau_H(G_0)=\infty$ if Chooser can avoid $H$ forever.
\end{definition}

In this manuscript we begin with the target $H=K_3$ and then discuss larger cliques.

\section{Triangle forcing}

\subsection{An exact one-round criterion}

For a graph $G$, define
\[
\sigma(G):=\max\{\nu(G[V(M)]-M): M \text{ is a matching in }G\},
\]
where $\nu(\cdot)$ denotes matching number.
Thus $\sigma(G)$ measures how many pairwise disjoint ``cross-edges'' can remain among
the endpoints of a matching chosen by Builder.

\begin{theorem}\label{thm:one-round-triangle}
Let $G$ be triangle-free. Then
\[
\tau_{K_3}(G)=1
\qquad\Longleftrightarrow\qquad
\sigma(G)\ge 2.
\]
\end{theorem}

\begin{proof}
Suppose Builder chooses a matching $M$ in $G$ and forms the intermediate graph $G^+$
with new vertex $w$.
Since $G$ is triangle-free, every triangle in $G^+$ must contain $w$.
Hence the triangles in $G^+$ are exactly the triples $wuv$
such that $uv\in E(G[V(M)]-M)$.

Therefore, edge-disjoint triangles in $G^+$ are in one-to-one correspondence with
pairwise disjoint edges in $G[V(M)]-M$.

If $\nu(G[V(M)]-M)\ge 2$, then $G^+$ contains two edge-disjoint triangles.
Builder places one entire triangle into the first part of the partition and the other
entire triangle into the second part. Whichever part Chooser keeps, a triangle survives.
Thus $\tau_{K_3}(G)=1$.

Conversely, if for every matching $M$ we have $\nu(G[V(M)]-M)\le 1$,
then after every Builder move the intermediate graph $G^+$ contains at most one triangle.
Hence at most one side of the partition can contain a triangle, and Chooser can always
keep the other side. Therefore Builder cannot force a triangle in one round.
\end{proof}

\begin{corollary}\label{cor:paths-oneround}
For the path $P_n$,
\[
\tau_{K_3}(P_n)=1
\qquad\Longleftrightarrow\qquad
n\ge 6.
\]
\end{corollary}

\begin{proof}
Let $P_n=v_1v_2\cdots v_n$.
If $n\ge 6$, choose
$ M=\{v_1v_2,\ v_3v_4,\ v_5v_6\}$.
Then $G[V(M)]-M$ contains the disjoint edges
$v_2v_3$ and $v_4v_5$,
so $\sigma(P_n)\ge 2$ and Theorem~\ref{thm:one-round-triangle} gives $\tau_H(P_n)=1$ (see Fig.~\ref{fig:dpg-step}).

If $n\le 5$, every matching $M$ has size at most $2$.
Thus $G[V(M)]$ has at most $4$ vertices, and since it is a subgraph of a path,
$G[V(M)]-M$ has matching number at most $1$.
Therefore $\sigma(P_n)\le 1$, and again by Theorem~\ref{thm:one-round-triangle},
$\tau_{K_3}(P_n)\neq 1$.
\end{proof}

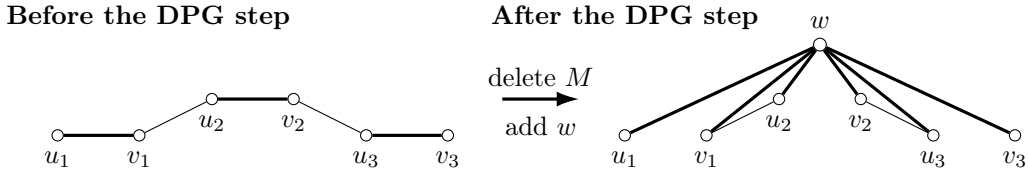
\begin{figure}[htbp]
\centering
\begin{tikzpicture}[scale=0.6,
    vertex/.style={circle,draw,fill=white,inner sep=1.5pt},
    newvertex/.style={circle,draw,fill=white,inner sep=1.7pt},
    matchedge/.style={very thick},
    edge/.style={thin},
    every node/.style={font=\small}
]

\begin{scope}[xshift=0cm]
\node at (2,2.6) {\textbf{Before the DPG step}};

\node[vertex,label=below:$u_1$] (u1) at (0,0) {};
\node[vertex,label=below:$v_1$] (v1) at (1.8,0) {};
\node[vertex,label=below:$u_2$] (u2) at (3.4,0.8) {};
\node[vertex,label=below:$v_2$] (v2) at (5.2,0.8) {};
\node[vertex,label=below:$u_3$] (u3) at (6.8,0) {};
\node[vertex,label=below:$v_3$] (v3) at (8.6,0) {};

\draw[edge] (v1)--(u2);
\draw[edge] (v2)--(u3);

\draw[matchedge] (u1)--(v1);
\draw[matchedge] (u2)--(v2);
\draw[matchedge] (u3)--(v3);

\node at (4.3,-1.0) {} ;
\end{scope}

\draw[-{Latex[length=3mm,width=2mm]}, very thick] (9.8,0.8) -- (11.5,0.8);
\node at (10.65,1.3) {\small delete $M$};
\node at (10.65,0.2) {\small add $w$};

\begin{scope}[xshift=12.5cm]
\node at (0,2.6) {\textbf{After the DPG step}};

\node[vertex,label=below:$u_1$] (U1) at (0,0) {};
\node[vertex,label=below:$v_1$] (V1) at (1.8,0) {};
\node[vertex,label=below:$u_2$] (U2) at (3.4,0.8) {};
\node[vertex,label=below:$v_2$] (V2) at (5.2,0.8) {};
\node[vertex,label=below:$u_3$] (U3) at (6.8,0) {};
\node[vertex,label=below:$v_3$] (V3) at (8.6,0) {};
\node[newvertex,label=above:$w$] (W) at (4.3,2.0) {};

\draw[edge] (V1)--(U2);
\draw[edge] (V2)--(U3);

\draw[matchedge] (W)--(U1);
\draw[matchedge] (W)--(V1);
\draw[matchedge] (W)--(U2);
\draw[matchedge] (W)--(V2);
\draw[matchedge] (W)--(U3);
\draw[matchedge] (W)--(V3);

\node at (4.3,-1.0) {};
\end{scope}

\end{tikzpicture}
\caption{A degree-preserving growth (DPG) step. Builder chooses the matching $M=\{u_1v_1,u_2v_2,u_3v_3\}$, deletes its edges, adds a new vertex $w$, and joins $w$ to all endpoints of $M$.}
\label{fig:dpg-step}
\end{figure}

\begin{theorem}\label{thm:cycles-triangle}
For the cycle $C_n$,
\[
\tau_{K_3}(C_n)=
\begin{cases}
0, & n=3,\\
1, & n=4,\\
2, & n=5,\\
1, & n\ge 6.
\end{cases}
\]
\end{theorem}

\begin{proof}
The case $n=3$ is immediate.

For $n\ge 6$, label the cycle as $v_1v_2\cdots v_nv_1$,
and choose $M=\{v_1v_2,\ v_3v_4,\ v_5v_6\}$. Then $G[V(M)]-M$ contains the disjoint edges $v_2v_3$ and $v_4v_5$,
so $\sigma(C_n)\ge 2$ and therefore $\tau_{K_3}(C_n)=1$
by Theorem~\ref{thm:one-round-triangle}.
The case $n=4$ also follows from this theorem.

It remains to handle $C_5$.
Label the cycle as $v_1v_2v_3v_4v_5v_1$.
We first show that $\tau_{K_3}(C_5)\neq 1$.
Indeed, every matching in $C_5$ has size at most $2$.
If Builder chooses a matching of size $1$, the intermediate graph contains no triangle.
If Builder chooses a matching of size $2$, say
$M=\{v_1v_2,\ v_3v_4\}$,
then the intermediate graph contains exactly one triangle, namely
$wv_2v_3$.
Hence Builder cannot force a triangle in one round.

We now show that $\tau_H(C_5)\le 2$.
Builder again chooses $M=\{v_1v_2,\ v_3v_4\}$.
The intermediate graph has edge set $
\{wv_1,wv_2,wv_3,wv_4,\ v_2v_3,\ v_4v_5,\ v_5v_1\}$.
This graph contains the triangle $T:=wv_2v_3$ and the $4$-cycle $Q:=wv_1v_5v_4w$.
These are edge-disjoint and together cover all edges of the intermediate graph.
Builder puts $E(T)$ into one part of the partition and $E(Q)$ into the other.
If Chooser keeps the first part, Builder has already forced a triangle.
If Chooser keeps the second part, the surviving graph is a $C_4$, from which Builder
wins in one more round.
Thus $\tau_{K_3}(C_5)=2$.
\end{proof}

\begin{remark}
Theorem~\ref{thm:cycles-triangle} identifies the $5$-cycle as the first genuinely
nontrivial two-round instance for triangle forcing.
\end{remark}

\section{Larger cliques: one-round forcing and support}

\subsection{A one-round criterion for larger cliques}

For a graph $H$ and an integer $r\ge 2$, let $\nu_{K_r}(H)$
denote the maximum number of pairwise vertex-disjoint copies of $K_r$ in $H$.
For a graph $G$, define
\[
\kappa_r(G):=\max\{\nu_{K_r}(G[V(M)]-M): M \text{ is a matching in }G\}.
\]

\begin{theorem}\label{thm:larger-cliques-oneround}
Let $r\ge 2$, and let $G$ be $K_{r+1}$-free.
Then
\[
\tau_{K_{r+1}}(G)=1
\qquad\Longleftrightarrow\qquad
\kappa_r(G)\ge 2.
\]
\end{theorem}

\begin{proof}
Fix a matching $M$ chosen by Builder, and let $G^+$ be the corresponding intermediate graph
with new vertex $w$.
Since $G$ is $K_{r+1}$-free and the DPG step adds no edges between old vertices, every copy of
$K_{r+1}$ in $G^+$ must contain the new vertex $w$.
Therefore such a copy is of the form $\{w\}\cup X$,
where $X$ is a set of $r$ old vertices spanning a copy of $K_r$ in $G[V(M)]-M$.
Conversely, every copy of $K_r$ in $G[V(M)]-M$ gives rise to a copy of $K_{r+1}$ in $G^+$
by joining to $w$.

Hence copies of $K_{r+1}$ in $G^+$ are in one-to-one correspondence with copies of $K_r$
in $G[V(M)]-M$.
Moreover, two such copies of $K_{r+1}$ are edge-disjoint if and only if the corresponding
copies of $K_r$ are vertex-disjoint.
If $\nu_{K_r}(G[V(M)]-M)\ge 2$ for some matching $M$, then $G^+$ contains two edge-disjoint
copies of $K_{r+1}$.
Builder places one whole copy into one part of the partition and the other whole copy into
the other part.
Whichever part Chooser keeps, it contains a copy of $K_{r+1}$.
Thus $\tau_{K_{r+1}}(G)=1$.

Conversely, if for every matching $M$ we have $\nu_{K_r}(G[V(M)]-M)\le 1$,
then every intermediate graph $G^+$ contains at most one copy of $K_{r+1}$.
Therefore at most one side of the partition can contain a copy of $K_{r+1}$, and Chooser
can always keep the other side.
Hence Builder cannot force $K_{r+1}$ in one round.
\end{proof}

\begin{corollary}\label{cor:no-one-round-k4}
If $G$ is triangle-free, then $\tau_{K_4}(G)\neq 1$.
\end{corollary}

\begin{proof}
If $G$ is triangle-free, then for every matching $M$, the graph $G[V(M)]-M$ is triangle-free.
Thus $\nu_{K_3}(G[V(M)]-M)=0$
for every $M$, so $\kappa_3(G)=0$.
Now apply Theorem~\ref{thm:larger-cliques-oneround} with $r=3$.
\end{proof}

\subsection{Supported cliques}

\begin{definition}
A \emph{supported clique of order $r$} in a graph $G$ is a pair $(C,S)$ such that
\begin{itemize}
    \item $C=\{c_1,\dots,c_r\}$ spans a copy of $K_r$ in $G$,
    \item $S=\{c_1x_1,\dots,c_rx_r\}$ is a matching in $G$,
    \item the support vertices $x_1,\dots,x_r$ all lie outside $C$.
\end{itemize}
\end{definition}

\begin{lemma}\label{lem:supported-promotion}
If $G$ contains a supported clique of order $r$, then Builder can create a copy of
$K_{r+1}$ in the intermediate graph in one round.
\end{lemma}

\begin{proof}
Let $
C=\{c_1,\dots,c_r\}$
be the clique, and let $
S=\{c_1x_1,\dots,c_rx_r\}$
be its support matching.
Builder chooses the matching $S$.
After the DPG step, the new vertex $w$ is adjacent to every $c_i$, while all clique edges
inside $C$ remain untouched.
Hence $G^+[\{w,c_1,\dots,c_r\}] \cong K_{r+1}$.
\end{proof}

\begin{remark}
Supported cliques give the natural one-step promotion mechanism for larger cliques.
The main difficulty is that the support edges are consumed by the DPG move, so recursive
forcing of large cliques appears to require a more elaborate template structure.
\end{remark}

\section{Exact forcing time for $K_4$}

\begin{definition}
A \emph{supported triangle} in a graph $G$ is a triangle
\[
T=\{a,b,c\}
\]
together with a matching
\[
S=\{ax,by,cz\}
\]
such that the vertices $x,y,z$ lie outside $T$.
\end{definition}

\begin{lemma}\label{lem:two-supported}
If a graph $G$ contains two vertex-disjoint (compatible) supported triangles, then Builder can force
a copy of $K_4$ in one round.
\end{lemma}

\begin{proof}
Let the two supported triangles be $T_1=\{a_1,b_1,c_1\}$ and $T_2=\{a_2,b_2,c_2\}$, 
with support matchings $S_1=\{a_1x_1,b_1y_1,c_1z_1\}$ and $S_2=\{a_2x_2,b_2y_2,c_2z_2\}$, respectively.
Since the two supported triangles are vertex-disjoint, the six support edges form a matching.

Builder chooses $M:=S_1\cup S_2$
as his matching.
Let $w$ be the new vertex created by the DPG step.
Then $\{w,a_1,b_1,c_1\}$ and $\{w,a_2,b_2,c_2\}$
span two edge-disjoint copies of $K_4$ in the intermediate graph.
Builder places one whole $K_4$ into the first part of the partition and the other whole
$K_4$ into the second part.
Whichever part Chooser keeps, it contains a copy of $K_4$.
\end{proof}

\begin{lemma}\label{lem:K4-lower}
If $G_0$ is triangle-free, then $\tau_{K_4}(G_0)\ge 3$.
\end{lemma}

\begin{proof}
First, Builder cannot force a copy of $K_4$ in one round by Corollary~\ref{cor:no-one-round-k4}.

Now consider Round~2.
Every triangle present in the graph after Round~1 contains $w_1$, because $G_0$ was
triangle-free and the first DPG step adds no old-old edges.
Suppose the second intermediate graph contains a copy of $K_4$.
It must contain the new vertex $w_2$, since the current graph before Round~2 is still $K_4$-free.
Removing $w_2$, the remaining three vertices form a triangle in the Round~1 graph, hence
must contain $w_1$.
Therefore every copy of $K_4$ in the second intermediate graph contains both $w_1$ and $w_2$.
So any two such copies share the edge $w_1w_2$, and hence cannot be edge-disjoint.

Therefore Builder cannot force a copy of $K_4$ in Round~2 either, since forcing in one round
requires two edge-disjoint copies of the target, one for each side of the partition.
Thus $\tau_{K_4}(G_0)\ge 3$.
\end{proof}

\begin{definition}
Let $H$ be the graph with vertex set
\[
\{z\}\cup \{a_i,b_i,a_i',a_i'',b_i',b_i'': i=1,2,3,4\},
\]
and edge set consisting of
\begin{itemize}
    \item the four triangles
    \[
    za_ib_i \qquad (i=1,2,3,4),
    \]
    \item the pendant edges
    \[
    a_ia_i',\ a_ia_i'',\ b_ib_i',\ b_ib_i'' \qquad (i=1,2,3,4).
    \]
\end{itemize}
Thus $H$ is a four-triangle fan around the apex $z$, with two leaves attached to each
non-apex vertex (Fig.~\ref{fig:gadget-H}).
\end{definition}

\begin{lemma}\label{lem:H}
From the graph $H$, Builder can force a graph containing two compatible supported
triangles in one round.
\end{lemma}

\begin{proof}
Builder chooses the matching $ M=\{a_3a_3',\ b_3b_3',\ a_4a_4',\ b_4b_4'\}$.
Let $z'$ be the new vertex created by the DPG step.

Then the intermediate graph contains the two old triangles
$T_0=\{z,a_1,b_1\}$ and $T_1=\{z,a_2,b_2\}$,
and the two new triangles $U_0=\{z',a_3,b_3\}$ and $U_1=\{z',a_4,b_4\}$.

Builder defines the first side of the partition to contain:
\begin{itemize}
    \item all edges of $T_0$, the support edges $a_1a_1'',b_1b_1'', za_4$, making $T_0$ a supported triangle;
    \item all edges of $U_0$, the support edges $a_3a_3'', b_3b_3'' z'a_4'$, making $U_0$ a supported triangle.
\end{itemize}

The second side is defined analogously:
\begin{itemize}
    \item all edges of $T_1$, the support edges $a_2a_2'', b_2b_2'',za_3$,
    making $T_1$ a supported triangle;
    \item all edges of $U_1$, the support edges $a_4a_4'', b_4b_4'', z'b_3'$,
    making $U_1$ a supported triangle.
\end{itemize}

These two subgraphs are edge-disjoint.
Hence whichever side Chooser keeps, the surviving graph contains two compatible
supported triangles (Fig.~\ref{fig:round2-round3-k4}).
\end{proof}

\begin{theorem}\label{thm:rho4}
There exists a triangle-free graph $G$ such that $\tau_{K_4}(G)=3$. Consequently, $\rhoK(4)=3$.
\end{theorem}

\begin{proof}
By Lemma~\ref{lem:K4-lower}, every triangle-free seed satisfies $\tau_{K_4}\ge 3.$
So it remains to construct a triangle-free graph $G$ with $\tau_{K_4}(G)\le 3$.

Let $C=c_1c_2\cdots c_{16}c_1$ be a $16$-cycle.
For each $i=1,\dots,16$, attach two leaves $c_i',\ c_i''$ to $c_i$.
Let $G$ be the resulting graph.
Clearly $G$ is triangle-free.

\medskip
\noindent
\textbf{Round 1.}
Builder chooses the matching
\[
M_1=\{c_1c_2,\ c_3c_4,\ c_5c_6,\ c_7c_8,\ c_9c_{10},\ c_{11}c_{12},\ c_{13}c_{14},\ c_{15}c_{16}\}.
\]
Let $w_1$ be the new vertex created by the DPG step.

The surviving cycle edges are
\[
c_2c_3,\ c_4c_5,\ c_6c_7,\ c_8c_9,\ c_{10}c_{11},\ c_{12}c_{13},\ c_{14}c_{15},\ c_{16}c_1,
\]
so the intermediate graph contains the eight triangles
\[
\{w_1,c_2,c_3\},\ \{w_1,c_4,c_5\},\ \dots,\ \{w_1,c_{16},c_1\}.
\]
Builder partitions the edge set so that one side contains exactly the first four such triangles,
together with all leaf edges incident to their non-apex vertices, and the other side contains
the last four such triangles with their leaf edges.
Whichever side Chooser keeps, the surviving graph is isomorphic to $H$ (Fig.~\ref{fig:gadget-H}).

\medskip
\noindent
\textbf{Round 2.}
From $H$, by Lemma~\ref{lem:H}, Builder can force a graph containing two compatible
supported triangles.

\medskip
\noindent
\textbf{Round 3.}
From two compatible supported triangles, by Lemma~\ref{lem:two-supported}, Builder
can force a copy of $K_4$.

Therefore $\tau_{K_4}(G)\le 3$.
\end{proof}

\begin{corollary}
With the seed from Theorem~\ref{thm:rho4}, we have $
\seed(4)\le 48.$
\end{corollary}

\begin{proof}
The seed consists of a $16$-cycle together with two leaves attached to each cycle vertex,
so it has $16+2\cdot 16=48$ vertices.
\end{proof}

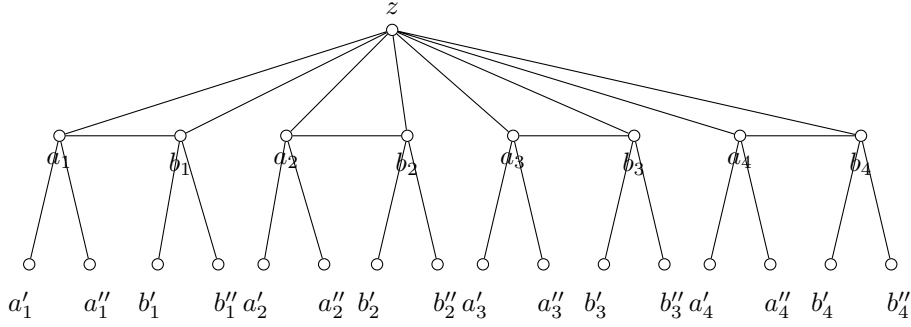
\begin{figure}[htbp]
\centering
\begin{tikzpicture}[scale=1.0,
    vertex/.style={circle,draw,fill=white,inner sep=1.5pt},
    edge/.style={thin},
    every node/.style={font=\small}
]

\node at (0,3.8) {\textbf{The four-triangle fan gadget $H$}};

\node[vertex,label=above:$z$] (z) at (0,2.3) {};

\node[vertex,label=below:$a_1$] (a1) at (-4.4,0.9) {};
\node[vertex,label=below:$b_1$] (b1) at (-2.8,0.9) {};
\draw[edge] (z)--(a1) -- (b1) -- (z);

\node[vertex,label={[xshift=-3pt,yshift=-4pt]below:$a_1'$}]   (a1p)  at (-4.8,-0.8) {};
\node[vertex,label={[xshift= 3pt,yshift=-4pt]below:$a_1''$}]  (a1pp) at (-4.0,-0.8) {};
\node[vertex,label={[xshift=-3pt,yshift=-4pt]below:$b_1'$}]   (b1p)  at (-3.1,-0.8) {};
\node[vertex,label={[xshift= 3pt,yshift=-4pt]below:$b_1''$}]  (b1pp) at (-2.3,-0.8) {};
\draw[edge] (a1)--(a1p);
\draw[edge] (a1)--(a1pp);
\draw[edge] (b1)--(b1p);
\draw[edge] (b1)--(b1pp);

\node[vertex,label=below:$a_2$] (a2) at (-1.4,0.9) {};
\node[vertex,label=below:$b_2$] (b2) at (0.2,0.9) {};
\draw[edge] (z)--(a2) -- (b2) -- (z);

\node[vertex,label={[xshift=-3pt,yshift=-4pt]below:$a_2'$}]   (a2p)  at (-1.7,-0.8) {};
\node[vertex,label={[xshift= 3pt,yshift=-4pt]below:$a_2''$}]  (a2pp) at (-0.9,-0.8) {};
\node[vertex,label={[xshift=-3pt,yshift=-4pt]below:$b_2'$}]   (b2p)  at (-0.2,-0.8) {};
\node[vertex,label={[xshift= 3pt,yshift=-4pt]below:$b_2''$}]  (b2pp) at (0.6,-0.8) {};
\draw[edge] (a2)--(a2p);
\draw[edge] (a2)--(a2pp);
\draw[edge] (b2)--(b2p);
\draw[edge] (b2)--(b2pp);

\node[vertex,label=below:$a_3$] (a3) at (1.6,0.9) {};
\node[vertex,label=below:$b_3$] (b3) at (3.2,0.9) {};
\draw[edge] (z)--(a3) -- (b3) -- (z);

\node[vertex,label={[xshift=-3pt,yshift=-4pt]below:$a_3'$}]   (a3p)  at (1.2,-0.8) {};
\node[vertex,label={[xshift= 3pt,yshift=-4pt]below:$a_3''$}]  (a3pp) at (2.0,-0.8) {};
\node[vertex,label={[xshift=-3pt,yshift=-4pt]below:$b_3'$}]   (b3p)  at (2.8,-0.8) {};
\node[vertex,label={[xshift= 3pt,yshift=-4pt]below:$b_3''$}]  (b3pp) at (3.6,-0.8) {};
\draw[edge] (a3)--(a3p);
\draw[edge] (a3)--(a3pp);
\draw[edge] (b3)--(b3p);
\draw[edge] (b3)--(b3pp);

\node[vertex,label=below:$a_4$] (a4) at (4.6,0.9) {};
\node[vertex,label=below:$b_4$] (b4) at (6.2,0.9) {};
\draw[edge] (z)--(a4) -- (b4) -- (z);

\node[vertex,label={[xshift=-3pt,yshift=-4pt]below:$a_4'$}]   (a4p)  at (4.2,-0.8) {};
\node[vertex,label={[xshift= 3pt,yshift=-4pt]below:$a_4''$}]  (a4pp) at (5.0,-0.8) {};
\node[vertex,label={[xshift=-3pt,yshift=-4pt]below:$b_4'$}]   (b4p)  at (5.8,-0.8) {};
\node[vertex,label={[xshift= 3pt,yshift=-4pt]below:$b_4''$}]  (b4pp) at (6.6,-0.8) {};
\draw[edge] (a4)--(a4p);
\draw[edge] (a4)--(a4pp);
\draw[edge] (b4)--(b4p);
\draw[edge] (b4)--(b4pp);

\end{tikzpicture}
\caption{The gadget $H$ forced after Round~1 in the proof of $\rhoK(4)=3$. It consists of four triangles sharing an apex $z$, together with two leaves attached to each non-apex vertex.}
\label{fig:gadget-H}
\end{figure}

\begin{figure}[htbp]
\centering
\begin{tikzpicture}[scale=1.0,
    vertex/.style={circle,draw,fill=white,inner sep=1.5pt},
    edge/.style={thin},
    chosen/.style={very thick},
    support/.style={dashed, very thick},
    every node/.style={font=\small}
]

\begin{scope}[yshift=4.8cm]
\node at (1.2,4.2) {\textbf{Round 2 on $H$}};

\node[vertex,label=above:$z$] (z) at (1.2,2.8) {};
\node[vertex,label=below:$a_1$] (a1) at (-2.4,1.2) {};
\node[vertex,label=below:$b_1$] (b1) at (-1.2,1.2) {};
\node[vertex,label=below:$a_2$] (a2) at (0.2,1.2) {};
\node[vertex,label=below:$b_2$] (b2) at (1.4,1.2) {};
\node[vertex,label=below:$a_3$] (a3) at (2.8,1.2) {};
\node[vertex,label=below:$b_3$] (b3) at (4.0,1.2) {};
\node[vertex,label=below:$a_4$] (a4) at (5.4,1.2) {};
\node[vertex,label=below:$b_4$] (b4) at (6.6,1.2) {};

\draw[edge] (z)--(a1)--(b1)--(z);
\draw[edge] (z)--(a2)--(b2)--(z);
\draw[edge] (z)--(a3)--(b3)--(z);
\draw[edge] (z)--(a4)--(b4)--(z);

\node[vertex,label={[xshift=-3pt,yshift=-4pt]below:$a_3'$}]  (a3p)  at (2.4,-0.6) {};
\node[vertex,label={[xshift= 3pt,yshift=-4pt]below:$a_3''$}] (a3pp) at (3.2,-0.6) {};
\node[vertex,label={[xshift=-3pt,yshift=-4pt]below:$b_3'$}]  (b3p)  at (3.8,-0.6) {};
\node[vertex,label={[xshift= 3pt,yshift=-4pt]below:$b_3''$}] (b3pp) at (4.6,-0.6) {};

\node[vertex,label={[xshift=-3pt,yshift=-4pt]below:$a_4'$}]  (a4p)  at (5.2,-0.6) {};
\node[vertex,label={[xshift= 3pt,yshift=-4pt]below:$a_4''$}] (a4pp) at (6.0,-0.6) {};
\node[vertex,label={[xshift=-3pt,yshift=-4pt]below:$b_4'$}]  (b4p)  at (6.6,-0.6) {};
\node[vertex,label={[xshift= 3pt,yshift=-4pt]below:$b_4''$}] (b4pp) at (7.4,-0.6) {};

\draw[edge] (a3)--(a3p); \draw[edge] (a3)--(a3pp);
\draw[edge] (b3)--(b3p); \draw[edge] (b3)--(b3pp);
\draw[edge] (a4)--(a4p); \draw[edge] (a4)--(a4pp);
\draw[edge] (b4)--(b4p); \draw[edge] (b4)--(b4pp);

\draw[chosen] (a3)--(a3p);
\draw[chosen] (b3)--(b3p);
\draw[chosen] (a4)--(a4p);
\draw[chosen] (b4)--(b4p);

\node[align=center] at (2.1,-1.8)
{Builder chooses 
$M_2=\{a_3a_3',\,b_3b_3',\,a_4a_4',\,b_4b_4'\}$.};
\end{scope}

\draw[-{Latex[length=3mm,width=2mm]}, very thick] (2.0,2.4) -- (2.0,1.2);
\node at (4,1.8) {\small DPG step + partition};

\begin{scope}[yshift=-3.6cm]
\node at (1.2,4.0) {\textbf{One surviving side after Round 2}};

\node[vertex,label=above:$z$] (Z) at (-1.8,2.5) {};
\node[vertex,label=below:$a_1$] (A1) at (-3.0,1.1) {};
\node[vertex,label=below:$b_1$] (B1) at (-1.8,1.1) {};
\draw[edge] (Z)--(A1)--(B1)--(Z);

\node[vertex,label={[xshift=-2pt,yshift=-4pt]below:$a_1''$}] (A1pp) at (-3.5,-0.4) {};
\node[vertex,label={[xshift= 2pt,yshift=-4pt]below:$b_1''$}] (B1pp) at (-1.2,-0.4) {};
\node[vertex,label=below:$a_4$] (A4aux) at (-0.2,1.1) {};

\draw[support] (A1)--(A1pp);
\draw[support] (B1)--(B1pp);
\draw[support] (Z)--(A4aux);

\node at (-1.8,-1.5) {$T_0$ with a support matching};

\node[vertex,label=above:$z'$] (Zp) at (4.2,2.5) {};
\node[vertex,label=below:$a_3$] (A3) at (3.0,1.1) {};
\node[vertex,label=below:$b_3$] (B3) at (4.2,1.1) {};
\draw[edge] (Zp)--(A3)--(B3)--(Zp);

\node[vertex,label={[xshift=-2pt,yshift=-4pt]below:$a_3''$}] (A3pp) at (2.5,-0.4) {};
\node[vertex,label={[xshift= 2pt,yshift=-4pt]below:$b_3''$}] (B3pp) at (4.8,-0.4) {};
\node[vertex,label=below:$a_4'$] (A4paux) at (5.8,1.1) {};

\draw[support] (A3)--(A3pp);
\draw[support] (B3)--(B3pp);
\draw[support] (Zp)--(A4paux);

\node at (4.2,-1.5) {$U_0$ with a support matching};

\end{scope}

\end{tikzpicture}
\caption{Round~2 in the proof of $\rhoK(4)=3$. From the fan gadget $H$, Builder chooses a matching on the third and fourth branches. After the DPG step and the partition, one surviving side contains two compatible supported triangles. In Round~3, Builder chooses the union of the two support matchings; the new apex then completes two edge-disjoint copies of $K_4$, one on each triangle core.}
\label{fig:round2-round3-k4}
\end{figure}
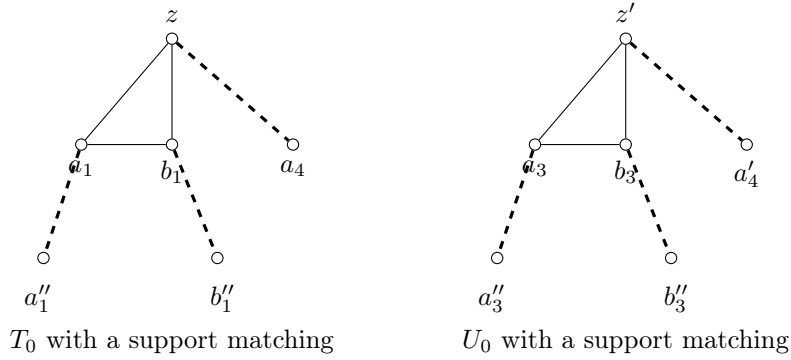

\section{Lower bounds for larger clique forcing}

\begin{lemma}\label{lem:omega-growth}
If $G^+$ is obtained from a graph $G$ by one DPG step, then
\[
\omega(G^+)\le \omega(G)+1.
\]
Consequently, along every play
\[
\omega(G_t)\le \omega(G_0)+t
\qquad\text{for all }t\ge 0.
\]
\end{lemma}

\begin{proof}
A DPG step adds a single new vertex $w$ and only edges incident with $w$.
Hence any clique in $G^+$ either avoids $w$ and was a clique in $G$, or it contains $w$. 
If it contains $w$, the remaining vertices of the clique must form a complete graph in $G[V(M)]-M$. 
Since $G[V(M)]-M$ is a subgraph of $G$, these remaining vertices form a clique of size one less in $G$. 
Therefore, $\omega(G^+)\le \max\{\omega(G), \omega(G[V(M)]-M) + 1\} \le \omega(G)+1$.

The second statement follows immediately by induction on $t$, since Chooser's partition step only deletes edges and thus cannot increase the clique number.
\end{proof}

\begin{lemma}\label{lem:all-new-vertices}
Let $r\ge 3$, and suppose that $G_0$ is $K_r$-free.
For $t\ge 1$, every copy of $K_{r-1+t}$
in the intermediate graph of round $t$ contains all new vertices $w_1,\dots,w_t$.
The same is therefore true in the current graph $G_t$.
\end{lemma}

\begin{proof}
We argue by induction on $t$.
For $t=1$, every copy of $K_r$ created in the first intermediate graph must contain $w_1$,
since $G_0$ is $K_r$-free.

Assume the statement holds for $t-1$, and consider round $t$.
By Lemma~\ref{lem:omega-growth},
\[
\omega(G_{t-1})\le (r-1)+(t-1)=r+t-2.
\]
Hence $G_{t-1}$ contains no copy of $K_{r-1+t}$.
Let $Q$ be a copy of $K_{r-1+t}$ in the intermediate graph of round $t$.
Then $Q$ must contain the new vertex $w_t$, since otherwise it would already lie in $G_{t-1}$.
Removing $w_t$, the remaining vertices span a copy of $K_{r-2+t}=K_{r-1+(t-1)}$
in $G_{t-1}$.
By the induction hypothesis, that clique contains $w_1,\dots,w_{t-1}$.
Therefore $Q$ contains all $w_1,\dots,w_t$.
\end{proof}

\begin{theorem}\label{thm:clique-lower-bound}
Let $r\ge 3$, and let $G_0$ be $K_r$-free.
Then for every $k\ge r+1$,
\[
\tau_{K_k}(G_0)\ge k-r+2.
\]
\end{theorem}

\begin{proof}
Set $t:=k-r+1$. By Lemma~\ref{lem:omega-growth}, the clique number cannot reach $k$ before round $t$.
Suppose Builder tries to force a copy of $K_k$ in round $t$.
By Lemma~\ref{lem:all-new-vertices}, every copy of $K_k$ in the intermediate graph of round $t$
contains all vertices $w_1,\dots,w_t$.
Since $k\ge r+1$, we have $t\ge 2$.
Hence any two such copies share at least one edge, namely an edge joining two of the vertices
$w_1,\dots,w_t$.
Therefore there are no two edge-disjoint copies of $K_k$ in the intermediate graph of round $t$.
Builder cannot partition the edge set so that both sides contain a copy of $K_k$.
Thus at least one further round is needed, and
\[
\tau_{K_k}(G_0)\ge t+1 = k-r+2.
\]
\end{proof}

\begin{corollary}\label{cor:trianglefree-k-lb}
If $G_0$ is triangle-free, then for every $k\ge 4$,
\[
\tau_{K_k}(G_0)\ge k-1.
\]
In particular,
\[
\tau_{K_4}(G_0)\ge 3,
\qquad
\tau_{K_5}(G_0)\ge 4.
\]
\end{corollary}

\begin{proof}
Apply Theorem~\ref{thm:clique-lower-bound} with $r=3$.
\end{proof}

\section{A conditional universality program}

\subsection{Quantitative parameters}

Recall that, for $k\ge 3$, $\rhoK(k)$ measures the minimum forcing time and $\seed(k)$ the minimum size of a triangle-free seed from which Builder can force a copy of $K_k$, with the convention that these quantities are $\infty$ if the corresponding class is empty

\begin{corollary}\label{cor:rho-lb}
We have
\[
\rhoK(3)=1,
\]
and for every $k\ge 4$,
\[
\rhoK(k)\ge k-1.
\]
\end{corollary}

\begin{proof}
The equality $\rhoK(3)=1$ follows from the fact that Builder can force a triangle from $C_4$
in one round.
For $k\ge 4$, apply Corollary~\ref{cor:trianglefree-k-lb}.
\end{proof}

\subsection{Universality conjecture}

\begin{conjecture}[Clique universality]\label{conj:universality}
For every $k\ge 3$,
\[
\rhoK(k)<\infty.
\]
Equivalently, every clique can be forced from some triangle-free seed.
\end{conjecture}

A natural approach to Conjecture~\ref{conj:universality} is to promote forceable cliques recursively. The proof of $\rhoK(4)=3$ suggests that the relevant inductive objects are not plain cliques, but supported copies of cliques.

\subsection{A supported-clique amplification mechanism}

\begin{definition}
A \emph{supported copy of $F$} in a graph $G$ is a pair $(F',S)$ such that
\begin{itemize}
    \item $F'$ is a subgraph of $G$ isomorphic to $F$,
    \item $S$ is a matching in $G$ such that every vertex of $F'$ is incident with exactly one edge of $S$,
    \item the endpoints of $S$ outside $F'$ are disjoint from $V(F')$.
\end{itemize}
Two supported copies $(F'_1,S_1)$ and $(F'_2,S_2)$ are \emph{compatible} if
the subgraphs $F'_1$ and $F'_2$ are vertex-disjoint and $S_1\cup S_2$ is a matching.
\end{definition}

The next lemma is the basic inductive step for clique forcing.

\begin{lemma}\label{lem:general-supported}
If a graph $G$ contains two compatible supported copies of $K_r$, then Builder can force
a copy of $K_{r+1}$ in one round.
\end{lemma}

\begin{proof}
Let the two copies be $C_1$ and $C_2$, with support matchings $S_1$ and $S_2$.
Since they are compatible, the union
\[
M:=S_1\cup S_2
\]
is a matching. Builder chooses $M$ for the DPG step and with the new vertex $w$.

For each $i\in\{1,2\}$, every vertex of $C_i$ is an endpoint of an edge of $M$, hence $w$
is adjacent to every vertex of $C_i$. Since the internal edges of $C_i$ are untouched,
\[
\{w\}\cup V(C_i)
\]
spans a copy of $K_{r+1}$ in the intermediate graph.
Because $C_1$ and $C_2$ are vertex-disjoint, these two copies of $K_{r+1}$ share only the vertex $w$,
and are therefore edge-disjoint.
Builder places one of them into each side of the partition, so whichever side Chooser keeps,
a copy of $K_{r+1}$ survives.
\end{proof}

This lemma reduces clique forcing to the problem of amplifying a forceable clique into two
compatible supported copies of itself.

\begin{conjecture}[Supported-clique amplification]\label{conj:amplifier}
For every $r\ge 3$, if there exists a triangle-free graph $G$ such that
\[
\tau_{K_r}(G)<\infty,
\]
then there exists a triangle-free graph $G'$ such that Builder can force two compatible
supported copies of $K_r$ from $G'$.
\end{conjecture}

The proof of $\rhoK(4)=3$ may be viewed as establishing Conjecture~\ref{conj:amplifier}
in the first nontrivial case $r=3$.

\begin{theorem}[Conditional universality]\label{thm:cond-univ}
Assume Conjecture~\ref{conj:amplifier}. Then for every $k\ge 3$, there exists a triangle-free
graph from which Builder can force $K_k$. Consequently,
\[
\rhoK(k)<\infty
\qquad\text{and}\qquad
\seed(k)<\infty
\]
for all $k\ge 3$.
\end{theorem}

\begin{proof}
We argue by induction on $k$. The base case $k=3$ follows from
Theorem~\ref{thm:one-round-triangle}.

Assume that there exists a triangle-free graph $G$ such that
\[
\tau_{K_r}(G)<\infty
\]
for some $r\ge 3$.
By Conjecture~\ref{conj:amplifier}, there exists a triangle-free graph $G'$ from which Builder
can force two compatible supported copies of $K_r$.
Applying Lemma~\ref{lem:general-supported}, Builder can then force $K_{r+1}$ in one additional round.
Thus $K_{r+1}$ is forceable from a triangle-free seed.
\end{proof}

\begin{remark}
A stronger quantitative version of Conjecture~\ref{conj:amplifier}, providing explicit control
on the forcing time and on the size of the amplified seed, would yield corresponding upper bounds
for $\rhoK(k)$ and $\seed(k)$ by iteration. We do not pursue such quantitative estimates here.
\end{remark}

\begin{remark}
The idea of the proof of Theorem~\ref{thm:cond-univ} suggests that, in the degree-preserving Builder--Chooser game,
the more plausible super-exponential parameter is the minimum size of a forcing seed rather
than the minimum forcing time. In this sense the DPG setting seems complementary to the original
Builder--Chooser clique game of Pettie, Tardos, and Walczak, where Builder starts from the empty
graph and the complexity is naturally measured in rounds.
\end{remark}

\section{Concluding remarks and open problems}

In this note, we introduced a degree-preserving variant of the Builder--Chooser clique game. By forcing growth to occur via local matching replacements rather than unconstrained global attachment, the game isolates the threshold at which dense subgraphs can reliably emerge from sparse seeds under severe topological constraints.

\subsection{First milestones}

The first concrete milestones in this program can be:
\begin{enumerate}
    \item prove that $\rhoK(5)<\infty$;
    \item prove that $\rhoK(6)<\infty$;
    \item construct explicit triangle-free seeds witnessing these bounds;
    \item extract from those constructions the correct recursive template notion for general $k$.
\end{enumerate}

\subsection{Open problems}

The results of this paper suggest several natural directions for further study.

\begin{enumerate}
    \item \textbf{The template-amplifier conjecture.}
    Prove Conjecture~\ref{conj:amplifier}. This is the main missing ingredient in the recursive universality program for forcing arbitrary cliques.

    \item \textbf{The forcing-time function $\rhoK(k)$.}
    We have $\rhoK(3)=1$ and $\rhoK(4)=3$, while the general lower bound yields $\rhoK(k)\ge k-1$ for all $k\ge 4$.
    Is $\rhoK(k)$ finite for every $k$?
    More generally, determine the growth rate of $\rhoK(k)$.

    \item \textbf{The seed-size function $\seed(k)$.}
    Determine the minimum size of a triangle-free seed from which Builder can force a copy of $K_k$.
    In particular, improve the bounds on $\seed(4)$ and investigate the asymptotic behavior of $\seed(k)$.

    \item \textbf{High-girth seeds.}
    For $k,g\ge 3$, let $\rhoK(k,g)$ denote the minimum forcing time for $K_k$ from a graph of girth at least $g$.
    Is $\rhoK(k,g)$ finite for all $k$ and $g$?

    \item \textbf{Other target graphs.}
    Study forcing times for targets other than cliques, especially cycles and complete bipartite graphs.
    It would be interesting to develop one-round criteria in these settings analogous to Theorem~\ref{thm:one-round-triangle}.
\end{enumerate}

\subsection*{Declaration on the use of generative AI}
Generative AI tools were used in a limited manner during manuscript preparation for language
editing and organizational assistance. The authors verified all mathematical content and are
fully responsible for the final manuscript.

\end{document}